\documentclass[10pt,a4paper]{article}
\usepackage[T1]{fontenc}

\usepackage[utf8]{inputenc}
\usepackage{authblk}
\usepackage{amsmath,amssymb,amsthm,mathtools,esint,bm,bbm}%,mathrsfs
\usepackage{mathrsfs}
\usepackage{bookmark}
\usepackage{amsmath}  

\usepackage{color}
\usepackage{framed}
\usepackage{enumitem}
\allowdisplaybreaks[3]
\newtheorem{definition}{Definition}[section]
\newtheorem{theorem}[definition]{Theorem}
\newtheorem{lemma}[definition]{Lemma}

\newtheorem{remark}[definition]{Remark}
\numberwithin{equation}{section}
\providecommand{\abs}[1]{\left|#1\right|}%绝对值符号
\providecommand{\norm}[1]{\lVert#1\rVert}%范数符号
\providecommand{\paren}[1]{\left( #1 \right)}%适应行高的小括号
\providecommand{\baren}[1]{\left\{ #1 \right\}}%适应行高的大括号
\newcommand{\uep}{\ensuremath{u^\varepsilon}}
\newcommand{\uepi}{\ensuremath{u^{\varepsilon_i}}}

\newcommand{\bigqep}{\ensuremath{Q^\varepsilon}}

\renewcommand{\r}{{\mathbb{R}}}
\newcommand{\n}{{\mathbb{N}}}
\newcommand{\sd}{{\mathbb{S}^{d-1}}}
\newcommand{\rd}{{\mathbb{R}^d}}

\newcommand{\zd}{{\mathbb{Z}^d}}

\newcommand{\ard}{\ensuremath{L^2(\rd)}}

\newcommand{\qqrd}{\ensuremath{L^{q}(\rd)}}

\newcommand{\ardloc}{\ensuremath{L^2_{\mathrm{loc}}(\rd)}}

\newcommand{\varezero}{\ensuremath{\varepsilon\rightarrow0}}

\newcommand{\hnorm}{\ensuremath{H^{\alpha/2}(\rd)}}

\newcommand{\hsp}{\hspace*{0.25cm}}%空格

\newcommand{\rmd}{\ensuremath{\,\mathrm{d}}}%积分全微分符号d
\newcommand{\ccinfty}{\ensuremath{C_c^\infty(\rd)}}

\newcommand{\linftyrd}{\ensuremath{L^\infty(\rd)}}

\newcommand{\lvar}{\ensuremath{L^\varepsilon}}

\newcommand{\intrd}{\ensuremath{\int_{\mathbb{R}^d}}}%Rd上积分
\newcommand{\iintrdd}{\ensuremath{\iint_{\mathbb{R}^d\times\mathbb{R}^d}}}%RdxRd上积分
\newcommand{\blv}{\ensuremath{\Lambda^\varepsilon}}

\newcommand{\ndv}{\ensuremath{N_\delta(\varepsilon)}}
\newcommand{\ikv}{\ensuremath{I_k(\varepsilon)}}

\DeclareMathOperator{\supp}{supp}
\newcommand{\esssup}{\operatorname*{ess\,sup}}%本性上界
\usepackage{hyperref}
\title{Periodic  homogenization of convolution type operators with irregular L\'{e}vy type tails}

\author[a]{Xiaofeng Jin}
\author[a]{Wentao Huo}
\author[b]{Lingwei Ma}
\author[b]{Zhenqiu Zhang\thanks{Corresponding author.}}

\affil[a]{School of Mathematical Sciences, Nankai University, Tianjin, 300071, P.~R.~China}
\affil[b]{School of Mathematical Sciences and LPMC, Nankai University, Tianjin, 300071, P.~R.~China}

\date{}
\begin{document}
\maketitle
\footnotetext[1]{E-mail: 1120220040@mail.nankai.edu.cn (X. Jin), 1120240051@mail.nankai.edu.cn
(W. Huo),
 malingwei@nankai.edu.cn (L. Ma), 
zqzhang@nankai.edu.cn (Z. Zhang).} 
\begin{abstract} 
We establish the homogenization results for a class of nonlocal operators of convolution type with integrable jumping kernel $p$ multiplied by rapidly oscillating periodic or locally periodic coefficients. The associated measure $p(z)\rmd z$ is assumed to belong to the domain of attraction of a symmetric $\alpha$-stable law.  We also assume that $p$ satisfies a pointwise L\'evy type lower bound and an averaged annular upper bound for points bounded away from the origin, and that the local $L^1$ oscillation of $p$ decays faster at infinity than its local $L^1$-norm. Under these assumptions, we prove the resolvent convergence of the nonlocal operators and explicitly determine the corresponding homogenized nonlocal operator, which is shown to be comparable to the fractional Laplacian. The proof relies on compactness arguments and a refined analysis based on the annular integral upper bound and an $\varepsilon$-cube decomposition.
\end{abstract}

\textbf{Keywords:} Periodic homogenization; nonlocal operators;
 oscillating coefficients;  singular kernels

\textbf{Mathematics Subject classification (2020):} 35B27, 47G10, 45P05, 45M05

\section{Introduction}
 Nonlocal  equations arise naturally in various applications  such as anomalous diffusion\cite{MR2722295}, image processing\cite{MR2480109}, mathematical finance\cite{MR1867081}, continuous particle systems\cite{MR2505861}, and population dynamics \cite{MR2120547, MR3327993}, and have attracted considerable attention in recent years.   Mathematically, many nonlocal equations studied in the literature are associated with operators of the form
\begin{equation}\label{eq,equation of general form}
  Lu(x)=\intrd K(x,y)\paren{u(y)-u(x)}\rmd y,
\end{equation}
where the kernel $K(x,y)$ satisfies suitable structural conditions. However, in many applications, the kernels $K$ exhibit fine-scale heterogeneities, reflecting complex microscopic structures in the underlying medium \cite{MR2870298,MR4132119,MR3554548}. 
Consequently, the associated operators contain coefficients that vary rapidly at the microscopic $\varepsilon$-scale, which makes the direct analysis of solutions to the corresponding nonlocal equations challenging. 
This difficulty motivates the study of their homogenization, namely, whether the large-scale behavior of solutions to nonlocal equations with rapidly oscillating coefficients can be effectively described by a simpler, "homogenized" equation as $\varepsilon \to 0$.

This nonlocal homogenization problem has been extensively studied in various settings. For instance, when the jumping kernel  is integrabel,
Piatnitski and Zhizhina \cite{MR3595876} considered the homogenization of the following diffusive scaling nonlocal operators with  convolution type kernels $p$ and rapidly oscillating coefficients $\blv$ in periodic media:
\begin{equation}\label{eq,operator 2017  Piatnitski and Zhizhina}
  \lvar_1 u(x)=\frac{1}{\varepsilon^{d+2}}\intrd p\big(\frac{x-y}{\varepsilon}\big)\blv(x,y)\paren{u(y)-u(x)}\rmd y,
\end{equation} 
where the kernel $p(z)\in L^1(\rd)\cap\ardloc$  is a nonnegative and even function with finite second moments. Using the corrector method, they proved that the family $\lvar_1$ converges to a local second-order elliptic operator $L^0_1$ with constant coefficient in the so-called $G$-topology or in the resolvent sense, that is for any $m>0$ the family of operators $(-\lvar_1+m)^{-1}$ converges strongly in $\ard$ to the operator $(-L^0_1+m)^{-1}$. In the follow-up work \cite{MR4053030},  they extended the initial approach for periodic media to a random framework and proved that the above result holds almost surely. 

On the other hand, when the kernel $K$ is singular, e.g., 
taking $K(z)$ the fractional Laplacian kernel $\abs{z}^{-d-\alpha}$ with $0<\alpha<2$, corresponding to the generator of a symmetric stable L\'evy process, Kassmann,  Piatnitski, and Zhizhina \cite{MR4000213} studied the homogenization of the operator 
\begin{equation}\label{eq,levy operator with oscillatory coefficient}
  \lvar_2 u(x)=\intrd \frac{u(y)-u(x)}{\abs{x-y}^{d+\alpha}}\blv(x,y)\rmd y.
\end{equation}
Here the periodic or random rapidly oscillating coefficients $\blv$ are assumed to have either product structure $\lambda\paren{x/\varepsilon}\mu\paren{y/\varepsilon}$ or symmetric structure
$\Lambda\paren{x,y,x/\varepsilon,y/\varepsilon}$ and to be uniformly positive and bounded.
Unlike \cite{MR3595876,MR4053030} or the homogenization of local operators\cite{MR4072212,MR4049392,MR4199276},  the limit operator in this setting is nonlocal and of L\'evy type. For degenerate and unbounded random coefficients,  \cite{MR4261110} established  homogenization results  and explicitly identified the corresponding limiting effective Dirichlet forms. Moreover, in \cite{Jin2026} we extended the results in 
 \cite{MR4000213} to the following nonlocal operator
  \begin{equation}\label{eq,main operator in jin2026+}
\lvar_3 u(x) = \int_{\mathbb{R}^d} K(x,y)\blv(x,y)(u(y)-u(x)) \rmd y 
\end{equation}
  with a broader class of singular kernel $K(x,y)$, allowing $K$ to be a singular measure or not of convolution type,  and in particular allowing stable-like kernels, i.e. $K(x-y)\asymp\abs{x-y}^{-d-\alpha}$ with $\alpha\in(0,2)$. We point out that  in \cite{Jin2026} if  $K$ is convolution type, then it is assumed to satisfy an averaged annular upper bound: 
 \begin{equation}\label{eq,upper condition}
   r^\alpha\int_{B_{2r}\setminus B_r}K(z)\rmd z\leq C
 \end{equation} 
 for some positive constant $C$ depending only on $d$ and $\alpha$ and for any $r>0$.
 Such a condition is common in the literature on the regularity of nonlocal integro-differential equations; see, for instance, \cite{MR4769823,MR4793281,MR4049224,MR3518535}.

Recently,  Piatnitski and Zhizhina \cite{Piatnitski2025} considered the homogenization of nonlocal equations associated with  convolution type operators 
\begin{equation}\label{eq:main operator}
  \lvar u(x)=\frac{1}{\varepsilon^{d+\alpha}}\intrd p\big(\frac{x-y}{\varepsilon}\big)\blv(x,y)\paren{u(y)-u(x)}\rmd y\quad \text{for}\quad x\in\rd,\quad\alpha\in(0,2),
\end{equation}  
where the kernel function $p(z)\in L^1(\rd)$ is non-negative and symmetric and the corresponding measure $p(z)\rmd z$ belongs to the domain of attraction of a stable law; see assumption \eqref{H3} for the definition and \cite[Chapter 8]{MR2780345} for the probabilistic background.  In order to obtain homogenization results,  they also assumed that  
$p(z)$ is stable-like for $\abs{z}\geq M$ and  that its local oscillation  decays at infinity faster than itself; see Remark \ref{remark,condition is weaker} for details.  Consequently, $p(z)$ has an infinite second moment,
which differs from the kernel condition  in \cite{MR3595876,MR4053030} and leads to a nonlocal effective operator.  Meanwhile, the stable-like condition in areas away from the origin also implies 
\begin{equation*}
  \frac{1}{\varepsilon^{d+\alpha}}p\big(\frac{x-y}{\varepsilon}\big)\asymp\abs{x-y}^{-d-\alpha}
\end{equation*}
for points $(x,y)$ away from the diagonal $x=y$ and for sufficiently small $\varepsilon>0$. This is closely related to \eqref{eq,upper condition}, as discussed above. Motivated by  our previous work \cite{Jin2026}, in this paper, we  extend the results of \cite{Piatnitski2025} to a less restrictive setting. More precisely,
we  assume that $p(z)$ satisfies  an averaged upper bound  of the form \eqref{eq,upper condition} only for $\abs{z}\geq M$; see assumption \eqref{H2}. Simultaneously, we also weaken the pointwise negligible oscillation assumption to one that holds only in an integral sense; see assumption \eqref{H4}. We then establish the homogenization result in the resolvent sense for the associated nonlocal operators with periodic and locally periodic coefficients $\blv(x, y)$ under these weaker conditions. 

The weaker assumptions on the kernel $p$ considered here also introduce additional analytical challenges, since the pointwise upper bound  
\begin{equation*}
  \frac{1}{\varepsilon^{d+\alpha}}p\big(\frac{x-y}{\varepsilon}\big)\lesssim\abs{x-y}^{-d-\alpha}
\end{equation*} 
are no longer available.  Moreover, in general one cannot expect  $\frac{1}{\varepsilon^{d+\alpha}}p\big(\frac{x-y}{\varepsilon}\big)$ to be uniformly bounded in $L^2$ with respect to $\varepsilon$ on the set $G_1^\delta$ defined in Lemma \ref{lemma,convergence on rd times rd}, in contrast to Lemma 3.3 in \cite{Piatnitski2025} where this property is used.  Although this uniform  boundedness in $L^2$ are not available, the annular integral upper bound condition \eqref{eq,h2 right}  yields uniform boundedness in $L^1_{\mathrm{loc}}(\rd\setminus\{0\})$ for the rescaled kernel $\frac{1}{\varepsilon^{d+\alpha}}p\big(\frac{z}{\varepsilon}\big)$. Inspired by the work \cite{MR4958113} on stochastic homogenization of elliptic equations with $L^1$ random convolution potentials, this local $L^1$-bound combined with the convolution structure of $p$, enables us to apply an  $\varepsilon$-cubes decomposition of  $G_1^\delta$  as in \cite{Piatnitski2025} to obtain the required convergence on $G_1^\delta$. We also mention that  when applying the $\varepsilon$-cubes decomposition, the contribution from cubes intersecting the boundary of $G_1^\delta$ must be shown to be negligible as $\varezero$. More precisely, the integral of $\frac{1}{\varepsilon^{d+\alpha}}p\big(\frac{x-y}{\varepsilon}\big)$ over such boundary-associated regions vanishes in the limit, as established in equality \eqref{eq,errer of G1 and Ik} in the proof of Lemma \ref{lemma,convergence on rd times rd}. In contrast, this issue does not appear in \cite{Piatnitski2025}  since the stronger assumptions $p(z)\lesssim\abs{z}^{-d-\alpha}$ for $\abs{z}\geq M$ there circumvent this boundary effect.

 The rest of the paper is organized as follows. In Section \ref{ass and res}, we introduce the precise assumptions on the kernel $p$ and the coefficient $\blv$ and state the main homogenization results. Section \ref{section,existence and uniqueness} is devoted to the existence and uniqueness of solutions to  the original equation and the effective  equation. In Section \ref{section,proof}, we prove the main theorems. The section is divided into several subsections: first we establish the compactness results of $\uep$ in $\ardloc$; then we identify the effective equation; finally we establish the homogenization in $\ard$ and complete the proof for the periodic and locally periodic settings.

\section{Assumptions and main results}\label{ass and res}
\subsection[Assumptions on the kernel p]{Assumptions on the kernel $p$}
Throughout the paper, we suppose that the  kernel $p(z)$  satisfies the following conditions:
\begin{enumerate}[label=(\textbf{H}\arabic{enumi}),ref=\textbf{H}\arabic{enumi}]
\item \label{H1}\textbf{(Non-negativity, symmetry, and integrability)}
\begin{equation}\label{eq,h1}
  p(z)\geq0,\quad p(z)=p(-z)\; \text{for all}\; z\in\rd, \quad \intrd p(z)\rmd z=1.
\end{equation}
  \item \label{H2} \textbf{(Two-sided  bounds)} There exist positive constants $C_1$, $C_2$, and $M\geq 1$ such that  for almost every $z\in\rd$ with $\abs{z}\geq M$,   
  \begin{equation}\label{eq,h2 left}
 p(z)\geq C_1\abs{z}^{-d-\alpha},
\end{equation}
and for all $r\geq M$,
\begin{equation}\label{eq,h2 right}
   r^\alpha\int_{B_{2r}\setminus B_r}p(z)\rmd z\leq C_2.
\end{equation}
  \item \label{H3} \textbf{(Attraction of a symmetric $\alpha$-stable law)} There exist a continuous symmetric positive function $k$ : $\sd\to\mathbb{R}_+$, such that for all open subset $D\subset \sd$ with boundary Lebesgue measure $\abs{\partial D}=0$,
  \begin{equation}\label{eq,h3}
    \int_{\{z\in\rd:\abs{z}>n,\tilde{z}\in D\}}p(z)\rmd z\sim\frac{1}{\alpha n^\alpha}\int_{D} k(s)\rmd S, \quad n\to\infty,
  \end{equation}
  where $\tilde{z}=z/\abs{z}$ and the symbol "$\sim$" means that the ratio of the two sides of this formula tends to one as $n\to\infty$.
  \item \label{H4}\textbf{(Negligible local $L^1$-oscillation relative to the local average of $p$ at infinity)}    Let $Q:=[-1/2,1/2]^d$ be the unit cube in $\rd$ and let $Q_z:=\{x+z:x\in Q\}$. Then,  
  \begin{equation}\label{eq,h4}
 \phi(r):=\sup_{\substack{{\abs{z'-z}\leq\sqrt{d}}\\\abs{z}\geq r}}\frac{\fint_{Q_z}\abs{p(x)-\fint_{Q_{z'}}p(y)\rmd y}\rmd x}{\fint_{Q_z}p(x)\rmd x}\to 0 \quad \text{as}\quad r\to\infty.
  \end{equation}
\end{enumerate}
Hereinafter, we use $\fint_{U}h$ to denote the $L^1$ average of a function $h$ 
 over a measurable set $U\subset\rd$, $i.e.$ 
\begin{equation*}
  \fint_{U}h=\frac{1}{|U|}\int_{U}h
\end{equation*}
 and use $B_{R}$ to denote a ball in $\rd$ centered at the origin with a radius $R$.
\begin{remark}\label{remark,fractional Laplacian kernel satisfy}
    It is worth  pointing out that the fractional Laplacian kernel $\abs{z}^{-d-\alpha}$, corresponding to the generator of a symmetric stable L\'evy process, clearly satisfies the condition \eqref{eq,h2 right}. Further examples satisfying  the condition \eqref{eq,h2 right} can be found in \cite{MR4793281,MR3518535}
    \end{remark}
    \begin{remark}\label{remark,condition is weaker}
      We compare assumptions \eqref{H2} and \eqref{H4} with the corresponding assumptions in Piatnitski and Zhizhina \cite{Piatnitski2025}, where the kernel $p$ is assumed to satisfy that for almost all $z$ such that $\abs{z}\geq M$,
  \begin{equation}\label{eq,h2 origin}
    \beta_1\abs{z}^{-d-\alpha}\leq p(z)\leq \beta_2\abs{z}^{-d-\alpha}
  \end{equation}
  and there exists a constant $K>0$ such that
  \begin{equation}\label{eq,h4 origin}
    \Phi_K(r):=\esssup_{\substack{\abs{z'-z}\leq K\\\abs{z}\geq r}}\frac{\abs{p(z')-p(z)}}{p(z)}\to 0\quad\text{as}\quad r\to\infty.
  \end{equation}
  Observe that the relation \eqref{eq,h4 origin} holds independently of the choice of $K$. Moreover, for sufficiently large $r>0$ depending only on $d$, it is easily seen that
    \begin{equation*}
      \phi(r)\leq\sup_{\substack{{\abs{z'-z}\leq\sqrt{d}}\\\abs{z}\geq r}}\frac{\fint_{Q_z}\fint_{Q_{z'}}\abs{p(x)-p(y)}\rmd y\rmd x}{\fint_{Q_z}p(x)\rmd x}\leq \Phi_{2\sqrt{d}}(r/2).
    \end{equation*}
    Therefore, in view of this fact and Remark \ref{remark,fractional Laplacian kernel satisfy}, our assumptions (cf. \eqref{H2} and \eqref{H4}) are weaker than those (cf. \eqref{eq,h2 origin} and \eqref{eq,h4 origin} above) in  \cite{Piatnitski2025}. 
    \end{remark}

\subsection{Periodic coefficients: structural assumptions and homogenization}
 In terms of the rapidly oscillating coefficient $\blv(x,y)$, let us first consider the periodic setting. In this case, we assume that $\blv(x,y)$ takes a form of $\Lambda(x/\varepsilon, y/\varepsilon)$, where $\Lambda(\xi, \eta)$ satisfies the following conditions:
 \begin{equation}\label{eq:Lamda condition periodic}
   \Lambda(\xi, \eta)=\Lambda(\eta, \xi),\quad \Lambda(\xi+z, \eta+z)=\Lambda(\xi, \eta),\quad
   \gamma^{-1}\leq\Lambda(\xi, \eta)\leq\gamma
 \end{equation}
 for all $\xi$, $\eta\in\rd$, $z\in\zd$, and some constant $\gamma\geq 1$.
 
 In order to illustrate our main theorems, we first define the limit operator
   \begin{equation}\label{eq:form of limit operator}
     L^0u(x)=\intrd \Lambda^{\text{eff}}(x,y)\frac{u(y)-u(x)}{\abs{x-y}^{d+\alpha}}\rmd y,
   \end{equation}
   where  
   \begin{equation}\label{eq:effective coefficient}
     \Lambda^{\textup{eff}}(x,y)=\bar{\Lambda}k\big(\frac{x-y}{\abs{x-y}}\big),\quad
     \bar{\Lambda}=\iint_{[0,1]^{2d}}\Lambda(x,y)\rmd x\rmd y.
   \end{equation}
   Note that $\Lambda^{\text{eff}}(x,y)$ is a uniformly positive and bounded function on $\{(x,y)\in\rd\times\rd:x\neq y\}$ since $0<c\leq k(s)\leq C$ on $\sd$.
   
   The main homogenization result in periodic setting is the following theorem.
   \begin{theorem}\label{thm:periodic theorem}
     Assume that \eqref{H1}--\eqref{H4} and \eqref{eq:Lamda condition periodic}  hold. For a given constant $m>0$ and each $f\in\ard$, let $\uep\in\ard$ be the solution of equation
 $-\lvar\uep+m\uep=f$ and let $u_0\in\hnorm$ be the solution of equation
 $-L^0u_0+mu_0=f$ with $L^0$ defined in \eqref{eq:form of limit operator}--\eqref{eq:effective coefficient}. Then, as $\varezero$, $\uep$ converges to $u_0$ strongly in $\ard$.
   \end{theorem}
   \begin{remark}\label{remark,existence and uniqueness}
     The existence and uniqueness of the solution $\uep$ in $\ard$ is an immediate consequence of the fact that the operator $\lvar$ is  negative and self-adjoint in $\ard$. Alternatively, the existence and uniqueness of the limit (weak) solution $u_0$ in $\hnorm$ can be established via the direct method in the calculus of variations. See Section \ref{section,existence and uniqueness} for details.
   \end{remark}
   \subsection{Locally periodic coefficients: structural assumptions and homogenization}
  Next we consider the case where the coefficient $\blv$ is of locally periodic symmetric structure. Precisely, we assume
\begin{equation}\label{eq, lambda x y xi eta}
	\Lambda^{\varepsilon}(x,y)=\Lambda\left(x,y,x/\varepsilon,y/\varepsilon\right) 
\end{equation} 
with a function $\Lambda(x,y,\xi,\eta)$ that is  measurable in $(\xi,\eta)$ for all $(x,y)$, and satisfies periodicity:
\begin{equation}\label{eq,symmetric of locally periodic}
	\Lambda(x,y,\xi+z,\eta)=\Lambda(x,y,\xi,\eta), \;\;  \Lambda(x,y,\xi,\eta+z)=\Lambda(x,y,\xi,\eta)
\end{equation}
for all $x, y,\xi, \eta\in\rd$ and $z\in\zd$. Moreover, $\Lambda(x,y,\xi,\eta)$ satisfies the following conditions:
\begin{align}\label{eq,locally periodic}
	\begin{cases}
		\Lambda(x,y,\xi,\eta)=\Lambda(x,y,\eta,\xi),\\
		\gamma^{-1}\leq \Lambda(x,y,\xi,\eta)\leq \gamma
	\end{cases}
\end{align}
for some constant $\gamma\geq1$ and all $x,y,\xi,\eta\in \mathbb{R}^{d}$. We also need the function $\Lambda(x,y,\xi,\eta)$ to be equicontinuous in $(x,y)$ uniformly with respect to $(\xi,\eta)$; that is, there exists a modulus of continuity $\omega$, independent of $(\xi,\eta)$, such that
\begin{equation}\label{eq, modulus of continuity}
  \abs{\Lambda(x_1,y_1,\xi,\eta)-\Lambda(x_2,y_2,\xi,\eta)}\leq
  \omega\paren{\abs{x_1-x_2}+\abs{y_1-y_2}}
\end{equation}
for  all $x_1,x_2,y_1,y_2,\xi,\eta\in \mathbb{R}^{d}$.

 The corresponding homogenization result in locally periodic settings is the following:
 \begin{theorem}\label{thm:locally periodic theorem}
   Under the assumptions of Theorem \ref{thm:periodic theorem} with \eqref{eq:Lamda condition periodic} replaced by \eqref{eq, lambda x y xi eta}--\eqref{eq, modulus of continuity}, the conclusion of Theorem \ref{thm:periodic theorem}  still holds  with \begin{equation}\label{eq:effective coefficient for locally periodic}
     \Lambda^{\textup{eff}}(x,y)=\bar{\Lambda}(x,y)k\big(\frac{x-y}{\abs{x-y}}\big),\quad
     \bar{\Lambda}(x,y)=\iint_{[0,1]^{2d}}\Lambda(x,y,\xi,\eta)\rmd \xi\rmd \eta.
   \end{equation}
 \end{theorem}
 \section{Existence and uniqueness}\label{section,existence and uniqueness}
 In this section, we focus on the existence and uniqueness of the solutions $\uep$ and $u_0$ to the equations stated in Theorem \ref{thm:periodic theorem}. As explained in Remark \ref{remark,existence and uniqueness}, we employ different methods to prove these results. 
  These two methods are adopted from \cite{MR3595876} and \cite{Jin2026}, respectively.
  
  In what follows, we use $C$ to denote various positive constants, possibly different from line to line, whose dependence on other parameters will be clear from the context unless explicitly stated. We  also write $a\lesssim b$ instead of $a\leq Cb$ for uniform constant $C$ which does not need to be represented exactly and may change from line to line. We also use the notation $a\asymp b$, which means that $a$ and $b$ are comparable functions
in the sense that there exists a constant $C\geq1$ independent of $a$ and $b$ such that
$a/C\leq b\leq Ca$ in a given domain.
 \begin{lemma}\label{lemma,existence and uniqueness of uep}
   Let $\lvar$ be defined in \eqref{eq:main operator}.  Assume that \eqref{H1} and either \eqref{eq:Lamda condition periodic}  (without periodicity) or \eqref{eq,locally periodic} hold. Let $m>0$ and $f\in\ard$. Then, there exists a unique solution $\uep$ in $\ard$ of 
   \begin{equation}\label{eq,equation of uep}
     -\lvar\uep+m\uep=f.
   \end{equation}
   Moreover, we have the estimate
   \begin{equation}\label{eq,estimate of uep}
     \norm{\uep}_{\ard}\leq\frac{1}{m}\norm{f}_{\ard}.
   \end{equation}
 \end{lemma}  
 \begin{proof}
   We first show that the operator $\lvar$ is  nonpositive and self-adjoint in $\ard$. Let $u$, $v\in\ard$. By the symmetry of $p$ and $\blv$, it is easily seen that
   \begin{equation}\label{eq, symmetry of lvar}
     (\lvar u,v)=(u,\lvar v)=-\frac{1}{2\varepsilon^{d+\alpha}}\iintrdd p\big(\frac{x-y}{\varepsilon}\big)\blv(x,y)\paren{u(y)-u(x)}\paren{v(y)-v(x)}\rmd y\rmd x,
   \end{equation}
    where we denote by $(\cdot,\cdot)$ the inner product in $\ard$. In particular,
   \begin{equation}\label{eq,l u u nonpositive}
     (\lvar u,u)=-\frac{1}{2\varepsilon^{d+\alpha}}\iintrdd p\big(\frac{x-y}{\varepsilon}\big)\blv(x,y)\paren{u(y)-u(x)}^2\rmd y\rmd x\leq0.
   \end{equation}

   We also observe that $\lvar$ is bounded in $\ard$. Indeed, letting $u\in\ard$, 
   in view of \eqref{H1} and \eqref{eq:Lamda condition periodic} (or \eqref{eq,locally periodic}),  we have
   \begin{equation*}
     \begin{aligned}
     \norm{\lvar u}_{\ard}^2 &=\frac{1}{\varepsilon^{2d+2\alpha}}\intrd\intrd p\big(\frac{x-y}{\varepsilon}\big)\blv(x,y)\paren{u(y)-u(x)}\rmd y\\
     &\quad\times\intrd p\big(\frac{x-z}{\varepsilon}\big)\blv(x,z)\paren{u(z)-u(x)}\rmd z\rmd x  \\
      &\lesssim \frac{1}{\varepsilon^{2d+2\alpha}}\intrd\intrd p\paren{\frac{x-y}{\varepsilon}}\paren{\abs{u(y)}+\abs{u(x)}}\rmd y\intrd p\paren{\frac{x-z}{\varepsilon}}\paren{\abs{u(z)}+\abs{u(x)}}\rmd z\rmd x\\
     &\lesssim\frac{1}{\varepsilon^{2d+2\alpha}}\intrd p\paren{\frac{y}{\varepsilon}}\intrd p\paren{\frac{z}{\varepsilon}}
     \intrd\paren{\abs{u(x+y)}+\abs{u(x)}}\paren{\abs{u(x+z)}+\abs{u(x)}}\rmd x\rmd z\rmd y\\
     &\lesssim\frac{1}{\varepsilon^{2\alpha}}\norm{u}_{\ard}^2.
   \end{aligned}
   \end{equation*}
   Thus, by the spectral theorem, we have
   \begin{equation*}
     \norm{\paren{m-\lvar}^{-1}}_{\mathcal{L}\paren{\ard,\ard}}\leq\frac{1}{m}.
   \end{equation*}
   Therefore, equation \eqref{eq,equation of uep} has a unique solution $\uep=\paren{m-\lvar}^{-1}f\in\ard$ and $\uep$ satisfies the estimate \eqref{eq,estimate of uep}.
 \end{proof}
 
 For the purpose of homogenization, we need the limit function of the family 
$\{\uep\}_{\varepsilon>0}$(as $\varezero$) to be unique. This is ensured by the existence and uniqueness of the solution $u_0$ to the limiting equation $-L^0u_0+mu_0=f$. We first introduce the definition of solution to this equation.  
\begin{definition}
  Let the operator $L^0$ be defined in \eqref{eq:form of limit operator}. Given $m>0$ and $f\in\ard$, we say that $u\in\hnorm$ is a weak solution of the equation
  \begin{equation}\label{eq,equation of u0}
    -L^0u_0+mu_0=f,
  \end{equation}
  if
  \begin{equation}\label{eq,weak solution def}
    -\frac{1}{2}\iintrdd\Lambda^{\textup{eff}}(x,y)\frac{\paren{u_0(y)-u_0(x)}\paren{v(y)-v(x)}}
    {\abs{x-y}^{d+\alpha}}\rmd y\rmd x+m\intrd u_0(x)v(x)\rmd x=\intrd f(x)v(x)\rmd x
  \end{equation}
  for all $v\in\hnorm$.
\end{definition}  
\begin{lemma}\label{lemma, existence and uniqueness of u0}
  There exists a unique weak solution $u_0$ to the equation \eqref{eq,equation of u0}.
\end{lemma}
\begin{proof}
  The uniqueness is an immediate consequence of the linearity of equation \eqref{eq,equation of u0}  and the estimate
  \begin{equation}\label{eq,estimate of u0}
    \norm{u_0}_{\hnorm}\lesssim\norm{f}_{\ard}.
  \end{equation}
  To see this estimate, test equation \eqref{eq,equation of u0} against $u_0$ (i.e. take $v=u_0$ in \eqref{eq,weak solution def}) and observe that
   \begin{equation}\label{eq,equal norm of u0}
     \iintrdd\Lambda^{\textup{eff}}(x,y)\frac{\paren{u_0(y)-u_0(x)}^2}
    {\abs{x-y}^{d+\alpha}}\rmd y\rmd x\asymp[u_0]_{\hnorm}^2,
   \end{equation}
  which is due to the fact that $C^{-1}\leq \Lambda^{\textup{eff}}(x,y)\leq C$ for a.e. $(x,y)\in\rd\times\rd$ and for some uniform constant $C\geq 1$. 
  We thus obtain the following two estimates:
  \begin{equation*}
    [u_0]_{\hnorm}^2\lesssim (f,u_0)\quad\text{and}\quad[u_0]_{\ard}^2\lesssim (f,u_0),
  \end{equation*}
  which, together with the Cauchy-Schwarz inequality, give the desired result \eqref{eq,estimate of u0}.
  
  We next prove the existence of the solution $u_0$ via the direct method in the calculus of variations. To this end, let us define a functional on $\hnorm$ as follows:
  \begin{equation*}
    F(u):=(-L^0 u,u)+m(u,u)-2(f,u).
  \end{equation*}
  We claim that the functional $F$ is continuous and strictly convex on $\hnorm$. Indeed, similar to \eqref{eq, symmetry of lvar}, for all $u$, $v\in\hnorm$, we have
  \begin{equation*}
    (L^0 u,v)=(u,L^0 v)=-\frac{1}{2}\iintrdd\Lambda^{\textup{eff}}(x,y)\frac{\paren{u(y)-u(x)}\paren{v(y)-v(x)}}
    {\abs{x-y}^{d+\alpha}}\rmd y\rmd x.
  \end{equation*}
  Hence, by the Cauchy-Schwarz inequality, we obtain
  \begin{equation*}
      \abs{\paren{-L^0 u, v}}\lesssim[u]_{\hnorm}[v]_{\hnorm}<\infty,
     \end{equation*} 
  In addition, by the symmetry of $L^0$, a straightforward calculation shows that 
   \begin{align*}
    F(u+tv)-F(u)&= 2t\baren{(-L^0 u,v)+m(u,v)+(f,v)}+t^2\baren{(-L^0 v,v)+m(v,v)}
   \end{align*} 
   for $u$, $v\in\hnorm$ and $t\in\r$.
  From the above estimate,  it clearly implies that $F$ is continuous on $\hnorm$.
  Moreover, for $u$, $v\in\hnorm$ with $u\neq v$ and  $t\in(0,1)$, in view of \eqref{eq,equal norm of u0}, we have
   \begin{equation*}
   F(tu+(1-t)v)-tF(u)-(1-t)F(v)=t(t-1)\baren{(-L^0(u-v),u-v)+m(u-v,u-v)}<0,
   \end{equation*}
   which implies that $F$ is strictly convex on $\hnorm$.  
   
   We proceed to show that $F$ is coercive in the sense that if $\norm{u}_{\hnorm}\rightarrow+\infty$, then $F(u)\rightarrow+\infty$. 
     Combining \eqref{eq,equal norm of u0}  with Young's inequality, we  deduce from the definition of $F$ that
    \begin{equation}\label{eq,F is coercive}
    \begin{aligned}
 F(u)&\geq C_3[u]_{\hnorm}^2+C_4\norm{u}_{\ard}^2-C_5\norm{f}_{\ard}^2\\
      &\geq \min\{C_3/2, C_4/2\} \norm{u}_{\hnorm}^2-C_5\norm{f}_{\ard}^2,
      \end{aligned}
    \end{equation}
    where the constants $C_3$, $C_4$, $C_5$ are positive and independent of $u$.
    This implies the coercivity of $F$. 
    Therefore, by a standard argument (see \cite[Corollary 3.23]{MR2759829} or \cite[Chapter 2, Proposition 1.2]{MR1727362}) on the existence and uniqueness of minimizers for functionals on reflexive Banach spaces,  we conclude from 
     the continuity, the strict convexity, and the coercivity of $F$ that  the functional $F$ admits its minimum at a unique point $\tilde{u}\in\hnorm$.
       
     On the other hand, we need to prove that $\tilde{u}$ satisfies  equation \eqref{eq,equation of u0} and hence we obtain the existence of the solution to it. Indeed, since  $\tilde{u}$ is the unique minimizer of $F$, it satisfies that
   \begin{equation*}
     \frac{\rmd}{\rmd t}F(\tilde{u}+tv)\Big|_{t=0}=0\quad \text{for all}\quad v\in\hnorm.
   \end{equation*}
   By the definition of $F$, this is equivalent to the fact that $\tilde{u}$ satisfies \eqref{eq,weak solution def},
  i.e., $\tilde{u}$ is a solution of \eqref{eq,equation of u0}.
\end{proof}
\begin{remark}
  The existence and uniqueness of  solutions to equations \eqref{eq,equation of uep} and \eqref{eq,equation of u0} can alternatively be obtained from the Lax-Milgram theorem applied to the bilinear forms associated with the weak formulations of the nonlocal operators $m-\lvar$ and $m-L^0$. The continuity and coercivity of these bilinear forms can follow from the same assumptions ensuring the continuity, strict convexity, and coercivity of the corresponding energy functionals.
\end{remark}
\section{Proof of main results}\label{section,proof}
The goal of this section is to establish the convergence of $\uep$ to $u_0$ in $\ard$. Since the problem is set on the whole space $\rd$, the proof is divided into several steps. We first show that the family $\{\uep\}_{\varepsilon>0}$ is relatively compact in $\ardloc$. This allows us to identify the limit and obtain convergence in $\ardloc$. Finally, by establishing a suitable control at infinity, we upgrade the local convergence to convergence in $\ard$. Throughout Subsections \ref{subsection1}--\ref{subsection3}, we work under the assumptions of Theorem \ref{thm:periodic theorem}.
\subsection[Compactness in L2locRd]{Compactness in $\ardloc$}\label{subsection1}
In this subsection, we establish the relative compactness of the family $\{\uep\}_{\varepsilon>0}$  in $\ardloc$ (see Lemma \ref{lemma,l2loc compactness for uep} below). Throughout this subsection, we work under exactly the same assumptions as in \cite{Piatnitski2025}. Therefore, we only sketch the proof and refer to \cite[Lemmas 3.1 and 3.2]{Piatnitski2025} for the full details. The proof relies the following Kolmogorov–M. Riesz–Fr\'echet compactness theorem, which can be view as the $L^q$-version of the Ascoli–Arzel\`a
theorem.
\begin{lemma}\cite[Theorem 4.26]{MR2759829}\label{lemma,Kolmogorov compactness}
  Let $\mathcal{U}$ be a bounded set in $\qqrd)$ with $1\leq q<\infty$. Assume that 
  \begin{equation}\label{eq,Kolmogorov compactness}
    \lim_{\abs{h}\to0}\norm{u(x+h)-u(x)}_{\qqrd}=0\quad\text{uniformly in}\quad u\in\mathcal{U}.
  \end{equation}
  Then the closure of  $\mathcal{U}|_G$ in $L^q(G)$ is compact for any measurable set $G\subset\rd$ with finite measure. Here $\mathcal{U}|_G$ denotes the restrictions to $G$ of the functions in $\mathcal{U}$.
\end{lemma}
Combining the above compactness lemma with the estimate \eqref{eq,estimate of uep} for $\uep$ in Lemma \ref{lemma,existence and uniqueness of uep}, we can obtain the following compactness result.
\begin{lemma}\label{lemma,l2loc compactness for uep}
Assume that  \eqref{eq,h1}, \eqref{eq,h2 left}, and  \eqref{eq:Lamda condition periodic} 
 hold. Then, for any sequence $\varepsilon_i\to0$, the set  $\{\uepi\}_{i\in\mathbb{N}}$ is relatively compact in $\ardloc$. Moreover, any limit point of this family is an element of $\hnorm$.
\end{lemma}
\begin{proof}[Sketch of the proof]
  The boundedness of the set $\{\uepi\}_{i\in\mathbb{N}}$ in $\ard$ has been obtained in \eqref{eq,estimate of uep}, therefore, in view of Lemma \ref{lemma,Kolmogorov compactness}, it suffices to prove \eqref{eq,Kolmogorov compactness} to establish its relative compactness. To see this, observe first that $(-\lvar\uep,\uep)\geq0$ by \eqref{eq,l u u nonpositive} and $m(\uep,\uep)\geq0$. Hence, if we test equation 
  \eqref{eq,equation of uep} against $\uep$, it follows from the Cauchy-Schwarz inequality and the estimate \eqref{eq,estimate of uep} that
  \begin{equation}\label{eq,p uep estimate}
     (-\lvar u,u)\asymp\frac{1}{\varepsilon^{d+\alpha}}\iintrdd p\big(\frac{x-y}{\varepsilon}\big)\paren{\uep(y)-\uep(x)}^2\rmd y\rmd x\lesssim\norm{f}_{\ard}^2\leq C.
  \end{equation} 
  In view of the lower bound condition \eqref{eq,h2 left} for $p$, the above estimate yields 
  \begin{align*}
    &\quad\iint_{\{(x,y):\abs{x-y}>M\varepsilon\}}\frac{\paren{\uep(y)-\uep(x)}^2}{\abs{x-y}^{d+\alpha}}\rmd y\rmd x\\
    &\lesssim
    \frac{1}{\varepsilon^{d+\alpha}}\iint_{\{(x,y):\abs{x-y}>M\varepsilon\}} p\big(\frac{x-y}{\varepsilon}\big)\paren{\uep(y)-\uep(x)}^2\rmd y\rmd x\leq C.
  \end{align*} 
  This allows us to show the following estimate 
  \begin{equation}\label{eq,uep difference estimate}
    \intrd\paren{\uep(x+h)-\uep(x)}^2\rmd x\leq 
    \begin{cases}
       C\abs{h}^\alpha, & \mbox{if } \abs{h}\geq3M\varepsilon \\
       C\varepsilon^\alpha, & \mbox{if }\abs{h}\leq3M\varepsilon,
    \end{cases}
   \end{equation}
   where the uniform constant $C>0$ is independent of $h$ and $\varepsilon$.
   With the help of estimate \eqref{eq,uep difference estimate}, we can consider separately the cases $\varepsilon_i\leq \abs{h}/(3M)$ and $\varepsilon_i\geq \abs{h}/(3M)$, leading to
   \begin{equation*}
     \forall \kappa>0\; \exists\delta>0\quad\text{such that}\quad\sup\limits_{\abs{h}\leq \delta,\varepsilon_i\geq 0}\intrd\paren{\uepi(x+h)-\uepi(x)}^2\rmd x\leq C\kappa.
   \end{equation*}
   This is equivalent to \eqref{eq,Kolmogorov compactness} upon taking $\mathcal{U}=\{\uepi\}_{i\in\mathbb{N}}$ and $q=2$.
   
   To see the latter statement of Lemma \ref{lemma,l2loc compactness for uep},  observe that we have shown
   \begin{equation*}
     \iintrdd\mathbbm{1}_{\abs{x-y}>M\varepsilon}(x-y)\frac{\paren{\uep(y)-\uep(x)}^2}
     {\abs{x-y}^{d+\alpha}}\rmd y\rmd x\leq C.
   \end{equation*}
   Let us consider an arbitrary limit point $u$ of the sequence $\uepi$ as $\varepsilon_i\to0$.
   Then by a diagonal argument, for a subsequence,  $\uepi$ converges to $u$ a.e. in $\rd$. 
   Therefore, for a.e. $(x,y)\in\rd\times\rd$, as $\varepsilon_i\to0$,
   \begin{equation*}
     \mathbbm{1}_{\abs{x-y}>M\varepsilon}(x-y)\frac{\paren{\uep(y)-\uep(x)}^2}
     {\abs{x-y}^{d+\alpha}}
     \to\frac{\paren{u(y)-u(x)}^2}
     {\abs{x-y}^{d+\alpha}}.
   \end{equation*}
   By the Fatou lemma,  we derive
   \begin{equation*}
    \iintrdd\frac{\paren{u(y)-u(x)}^2}
     {\abs{x-y}^{d+\alpha}}\rmd y\rmd x\leq C,
   \end{equation*}
   which, together with the fact that $\norm{u}_{\ard}\leq C$, yields $u\in\hnorm$.
\end{proof}
\subsection{Identification of the limit}
In this subsection, we identify the limit function $u\in\hnorm$ by proving that $u$ is a weak solution (in the sense of \eqref{eq,weak solution def} to the limit equation \eqref{eq,equation of u0}. Thanks to the uniqueness result established in Lemma \ref{lemma, existence and uniqueness of u0}, this identification implies that the whole family $\uep$ converges to $u$ in $\ardloc$. Due to the density of $\ccinfty$ in $\hnorm$, we need to show that for any $\varphi\in\ccinfty$, $u$ satisfies the following equality: 
\begin{equation*}
  \frac{1}{2}\iintrdd\Lambda^{\textup{eff}}(x,y)\frac{\paren{u(y)-u(x)}\paren{\varphi(y)-\varphi(x)}}
    {\abs{x-y}^{d+\alpha}}\rmd y\rmd x+m\intrd u(x)\varphi(x)\rmd x=\intrd f(x)\varphi(x)\rmd x.
\end{equation*}
With this goal in mind, testing equation \ref{eq,equation of uep} against $\varphi$, we obtain
\begin{align*}
  &\frac{1}{2\varepsilon^{d+\alpha}}\iintrdd p\big(\frac{x-y}{\varepsilon}\big)\blv(x,y)\paren{\uep(y)-\uep(x)}\paren{\varphi(y)-\varphi(x)}\rmd y\rmd x\\
  &\quad+ m\intrd \uep(x)\varphi(x)\rmd x=\intrd f(x)\varphi(x)\rmd x.
\end{align*}
Observe that by Lemma \ref{lemma,l2loc compactness for uep}, for a subsequence, we have
\begin{equation*}
  \intrd \uep(x)\varphi(x)\rmd x\to\intrd u(x)\varphi(x)\rmd x, \quad\text{as}\quad\varezero.
\end{equation*}
Therefore, it suffices to prove the following lemma:
\begin{lemma}\label{lemma,convergence on rd times rd}
  Assume that \eqref{H1}--\eqref{H4} and \eqref{eq:Lamda condition periodic}--\eqref{eq:effective coefficient} hold. Then, 
  \begin{equation}\label{eq,convergence of main term}
  \begin{aligned}
    &\lim_{\varezero}\frac{1}{\varepsilon^{d+\alpha}}\iintrdd p\big(\frac{x-y}{\varepsilon}\big)\blv(x,y)\paren{\uep(y)-\uep(x)}\paren{\varphi(y)-\varphi(x)}\rmd y\rmd x\\&=\iintrdd\Lambda^{\textup{eff}}(x,y)\frac{\paren{u(y)-u(x)}\paren{\varphi(y)-\varphi(x)}}
    {\abs{x-y}^{d+\alpha}}\rmd y\rmd x.
    \end{aligned}
  \end{equation}
\end{lemma}
\begin{proof}
  The desired convergence follows by decomposing the integration domain $\rd\times\rd$ into three regions and passing to the limit on each of them. More precisely, for $\delta>0$, we define
  \begin{align*}
    G_1^\delta &:=\baren{(x,y)\in\rd\times\rd:\abs{x-y}\geq\delta,\abs{x}+\abs{y}\leq\delta^{-1}}, \\
     G_2^\delta & :=\baren{(x,y)\in\rd\times\rd:\abs{x-y}\leq\delta,\abs{x}+\abs{y}\leq\delta^{-1}},  \\
     G_3^\delta &:=\baren{(x,y)\in\rd\times\rd:\abs{x}+\abs{y}\geq\delta^{-1}}. 
  \end{align*}
  To simplify notation, we introduce the shorthand
  \begin{equation*}
    \bigqep(x,y):=\frac{1}{\varepsilon^{d+\alpha}} p\big(\frac{x-y}{\varepsilon}\big)\blv(x,y)\paren{\uep(y)-\uep(x)}\paren{\varphi(y)-\varphi(x)}
  \end{equation*}
  and
  \begin{equation*}
    Q(x,y):=\Lambda^{\textup{eff}}(x,y)\frac{\paren{u(y)-u(x)}\paren{\varphi(y)-\varphi(x)}}
    {\abs{x-y}^{d+\alpha}}.
  \end{equation*}
  Let us first prove 
  \begin{equation}\label{eq,limit on G2,G3}
    \lim_{\delta\to0}\lim_{\varezero}\iint_{G_2^\delta\cup G_3^\delta}\bigqep(x,y)\rmd y\rmd x=0.
  \end{equation}
  To see this, observe that by \eqref{eq:Lamda condition periodic}, the Cauchy-Schwarz inequality, and \eqref{eq,p uep estimate},  we have
  \begin{equation}\label{eq, estimate of integral on G2 and G3}
  \begin{aligned}
    \abs{\iint_{G_2^\delta\cup G_3^\delta}\bigqep(x,y)\rmd y\rmd x}
    &\lesssim\paren{\frac{1}{\varepsilon^{d+\alpha}}\iint_{G_2^\delta\cup G_3^\delta} p\big(\frac{x-y}{\varepsilon}\big)\paren{\uep(y)-\uep(x)}^2\rmd y\rmd x}^{1/2}\\
    &\quad \times\paren{\frac{1}{\varepsilon^{d+\alpha}}\iint_{G_2^\delta\cup G_3^\delta} p\big(\frac{x-y}{\varepsilon}\big)\paren{\varphi(y)-\varphi(x)}^2\rmd y\rmd x}^{1/2}\\
    &\lesssim \paren{\frac{1}{\varepsilon^{d+\alpha}}\iint_{G_2^\delta\cup G_3^\delta} p\big(\frac{x-y}{\varepsilon}\big)\paren{\varphi(y)-\varphi(x)}^2\rmd y\rmd x}^{1/2}.
    \end{aligned}
  \end{equation}
  We first estimate the last integral on $G_2^\delta$. Applying the fact that $\abs{\varphi(y)-\varphi(x)}\leq\norm{\nabla\varphi}_{\linftyrd}\abs{y-x}$ for $\varphi\in\ccinfty$ yields
  \begin{align*}
  &\quad\frac{1}{\varepsilon^{d+\alpha}}\iint_{G_2^\delta} p\big(\frac{x-y}{\varepsilon}\big)\paren{\varphi(y)-\varphi(x)}^2\rmd y\rmd x\\
  &=\frac{1}{\varepsilon^{d+\alpha}}\iint_{G_2^\delta\cap\{(x,y):\abs{x-y}\leq M\varepsilon\}} p\big(\frac{x-y}{\varepsilon}\big)\paren{\varphi(y)-\varphi(x)}^2\rmd y\rmd x \\
    &\quad +\frac{1}{\varepsilon^{d+\alpha}}\iint_{G_2^\delta\cap\{(x,y):\abs{x-y}\geq M\varepsilon\}} p\big(\frac{x-y}{\varepsilon}\big)\paren{\varphi(y)-\varphi(x)}^2\rmd y\rmd x  \\
    & \lesssim \frac{\norm{\nabla\varphi}_{\linftyrd}^2\abs{\supp \varphi}}{\varepsilon^{d+\alpha}}\paren{\varepsilon^2\int_{\{z:\abs{z}\leq M\varepsilon\}}p\paren{\frac{z}{\varepsilon}}\rmd z+\int_{\{z:M\varepsilon\leq\abs{z}\leq \delta\}}p\paren{\frac{z}{\varepsilon}}\abs{z}^2\rmd z}\\
    & \lesssim \norm{\nabla\varphi}_{\linftyrd}^2\abs{\supp \varphi}\varepsilon^{2-\alpha}\paren{\int_{\{z:\abs{z}\leq M\}}p(z)\rmd z+\int_{\{z:M\leq\abs{z}\leq \delta/\varepsilon\}}p(z)\abs{z}^2\rmd z}.
     \end{align*}
    By assumption \eqref{H1}, we have
    \begin{equation*}
      \int_{\{z:\abs{z}\leq M\}}p(z)\rmd z\leq 1.
    \end{equation*} 
   Since $\delta/\varepsilon\gg1$, we can choose $j\in\mathbb{N}$  such that $2^jM\leq \delta/\varepsilon< 2^{j+1}M$.
    From assumption \eqref{eq,h2 right}, it follows that
    \begin{equation}\label{eq, estimate pz z2}
    \begin{aligned}
     \int_{\{z:M\leq\abs{z}\leq \delta/\varepsilon\}}p(z)\abs{z}^2\rmd z
      & \leq \sum_{i=0}^{j}\int_{\{z:2^iM\leq\abs{z}\leq 2^{i+1}M\}}p(z)(2^{i+1}M)^2\rmd z \\
      & \lesssim \sum_{i=0}^{j}2^{(2-\alpha)i}\lesssim (\frac{\delta}{\varepsilon})^{2-\alpha}.
      \end{aligned} 
      \end{equation}
      Putting together the above estimates, we obtain
      \begin{equation}\label{eq,estimate on G2}
        \frac{1}{\varepsilon^{d+\alpha}}\iint_{G_2^\delta} p\big(\frac{x-y}{\varepsilon}\big)\paren{\varphi(y)-\varphi(x)}^2\rmd y\rmd x\leq C(\varepsilon^{2-\alpha}+\delta^{2-\alpha}).
      \end{equation}
      We proceed to show that the last integral on $G_3^\delta$ in \eqref{eq, estimate of integral on G2 and G3} is negligible as $\varezero$ and $\delta\to0$. Let
       \begin{equation*}
        O:= \baren{(x,y)\in\rd\times\rd:x\in \supp \varphi\;\text{or}\; y\in\supp \varphi}.
       \end{equation*}
       Observe that for sufficiently small $\delta>0$ (depending on $\supp \varphi$), the following relation holds:
      \begin{equation*}
        G_3^\delta\cap O\subset\baren{(x,y):\abs{x-y}>1/(2\delta)}\cap O.
      \end{equation*}
      Therefore, using a dyadic decomposition and  assumption \eqref{eq,h2 right} again leads to
      \begin{align*}
        &\frac{1}{\varepsilon^{d+\alpha}}\iint_{G_3^\delta} p\big(\frac{x-y}{\varepsilon}\big)\paren{\varphi(y)-\varphi(x)}^2\rmd y\rmd x
        \lesssim \frac{\norm{\varphi}_{\linftyrd}^2\abs{\supp \varphi}}{\varepsilon^{d+\alpha}}\int_{\{z:\abs{z}\geq 1/(2\delta)\}}p\paren{\frac{z}{\varepsilon}}\rmd z\\
        &\lesssim \varepsilon^{-\alpha}\int_{\{z:\abs{z}\geq 1/(2\delta\varepsilon)\}}p(z)\rmd z\leq\varepsilon^{-\alpha}\sum_{i=-1}^{+\infty}\int_{\{z:2^{i}/(\delta\varepsilon)\leq\abs{z}\leq
    2^{i+1}/(\delta\varepsilon)\}}p(z)\rmd z\lesssim \varepsilon^{-\alpha}(\delta\varepsilon)^\alpha=\delta^\alpha.
      \end{align*}
      Inserting this estimate and \eqref{eq,estimate on G2} into \eqref{eq, estimate of integral on G2 and G3} gives our desired result \eqref{eq,limit on G2,G3}.
      
      On the other hand, since $\Lambda^{\textup{eff}}(x,y)\leq C$ for a.e. $(x,y)\in\rd\times\rd$, by the Cauchy-Schwarz inequality, 
      \begin{equation*}
        \abs{\iint_{G_2^\delta\cup G_3^\delta}Q(x,y)\rmd y\rmd x}\lesssim [u]_{H^{\alpha/2}(G_{2}^{\delta}\cup G_{3}^{\delta})}[\varphi]_{H^{\alpha/2}(G_{2}^{\delta}\cup G_{3}^{\delta})}.
      \end{equation*}
     Observe that
      \begin{equation*}
    \mathbbm{1}_{G_2^\delta\cup G_3^\delta}(x,y)\to 0\hsp\text{as}\; \delta\to0
  \end{equation*}
   for a.e. $(x,y)\in\rd\times\rd$. The Lebesgue dominated convergence theorem  implies
   \begin{equation}\label{eq, limit of Q on G2 and G3}
     \lim_{\delta\to0}\iint_{G_2^\delta\cup G_3^\delta}Q(x,y)\rmd y\rmd x=0.
  \end{equation}
   Having established \eqref{eq,limit on G2,G3} and \eqref{eq, limit of Q on G2 and G3}, in order to conclude the lemma, it suffices to prove that 
   \begin{equation}\label{eq,convergence on G1}
     \lim_{\varezero}\iint_{G_1^\delta}\bigqep(x,y)\rmd y\rmd x=\iint_{G_1^\delta}Q(x,y)\rmd y\rmd x
   \end{equation}
   for any fixed small $\delta>0$.
   Due to the symmetry of the above under the exchange of $x$ and $y$, it is enough to establish the following two equalities:
  \begin{equation}\label{eq,only u}
    \lim_{\varezero}\frac{1}{\varepsilon^{d+\alpha}}\iint_{G_1^\delta} p\big(\frac{x-y}{\varepsilon}\big)\blv(x,y)\uep(x)\varphi(x)\rmd y\rmd x=\iint_{G_1^\delta}\Lambda^{\textup{eff}}(x,y)\frac{u(x)\varphi(x)}
    {\abs{x-y}^{d+\alpha}}\rmd y\rmd x,
  \end{equation}
  \begin{equation}\label{eq,u x and varphi y}
    \lim_{\varezero}\frac{1}{\varepsilon^{d+\alpha}}\iint_{G_1^\delta} p\big(\frac{x-y}{\varepsilon}\big)\blv(x,y)\uep(x)\varphi(y)\rmd y\rmd x=\iint_{G_1^\delta}\Lambda^{\textup{eff}}(x,y)\frac{u(x)\varphi(y)}
    {\abs{x-y}^{d+\alpha}}\rmd y\rmd x.
  \end{equation}
  We only prove \eqref{eq,u x and varphi y}, since the proof of \eqref{eq,only u} is analogous. We first claim that 
  \begin{equation}\label{eq,weak convergence of p and k}
  \frac{1}{\varepsilon^{d+\alpha}} p\big(\frac{z}{\varepsilon}\big) \to \frac{k\paren{z/\abs{z}}}{\abs{z}^{d+\alpha}} \quad\text{weakly in} \quad L^1_{\mathrm{loc}}(\rd\setminus\{0\})
  \end{equation}
  as $\varezero$.
  Indeed, for any compact set $E\Subset\rd\setminus\{0\}$ and $0<\varepsilon<r/M$ with $r:=\inf_{z\in E}\abs{z}>0$, we can use assumption \eqref{eq,h2 right} to obtain
  \begin{align}\label{eq,bound on E}
    \int_E\frac{1}{\varepsilon^{d+\alpha}} p\big(\frac{z}{\varepsilon}\big)\rmd z
    \leq\int_{\{z:\abs{z}\geq r\}}\frac{1}{\varepsilon^{d+\alpha}} p\big(\frac{z}{\varepsilon}\big)\rmd z \leq\frac{1}{\varepsilon^\alpha}\sum_{i=0}^{+\infty}\int_{\{z:2^{i}r/\varepsilon\leq\abs{z}\leq
    2^{i+1}r/\varepsilon\}}p(z)\rmd z\lesssim r^{-\alpha}.
  \end{align}
  This implies that $\frac{1}{\varepsilon^{d+\alpha}} p\big(\frac{z}{\varepsilon}\big)$ is bounded uniformly in $L^1(E)$. Moreover, for any $R_2>R_1>0$ and an open set $D\subset \sd$, it follows from assumption \eqref{H3} that 
  \begin{align*}
&\int_{\{z:R_1<\abs{z}<R_2,\tilde{z}\in D\}}\frac{1}{\varepsilon^{d+\alpha}} p\big(\frac{z}{\varepsilon}\big)\rmd z   =\int_{\{z:\frac{R_1}{\varepsilon}<\abs{z}<\frac{R_2}{\varepsilon},\tilde{z}\in D\}}\frac{1}{\varepsilon^\alpha}p(z)\rmd z\\&\to\frac{1}{\alpha}\paren{\frac{1}{R_1^\alpha}-\frac{1}{R_2^\alpha}}\int_{\tilde{z}\in D}k(s)\rmd S
    =\int_{R_1}^{R_2}\frac{\rmd r}{r^{\alpha+1}}\int_{\tilde{z}\in D}k(s)\rmd S\\
    &=\int_{\{z:R_1<\abs{z}<R_2,\tilde{z}\in D\}}\frac{k\paren{z/\abs{z}}}{\abs{z}^{d+\alpha}}\rmd z
  \end{align*}
  as $\varezero$.
  Hence, the claim \eqref{eq,weak convergence of p and k} is proved. Next, we apply this claim to prove
  \begin{equation}\label{eq,limit 1}
    \lim_{\varezero}\frac{1}{\varepsilon^{d+\alpha}}\iint_{G_1^\delta} p\big(\frac{x-y}{\varepsilon}\big)u(x)\varphi(y)\rmd y\rmd x=\iint_{G_1^\delta}k\big(\frac{x-y}{\abs{x-y}}\big)\frac{u(x)\varphi(y)}
    {\abs{x-y}^{d+\alpha}}\rmd y\rmd x.
  \end{equation} 
  By the change of variables $z=x-y$, we have
  \begin{align*}
    &\abs{\frac{1}{\varepsilon^{d+\alpha}}\iint_{G_1^\delta} p\big(\frac{x-y}{\varepsilon}\big)u(x)\varphi(y)\rmd y\rmd x-\iint_{G_1^\delta}k\big(\frac{x-y}{\abs{x-y}}\big)\frac{u(x)\varphi(y)}
    {\abs{x-y}^{d+\alpha}}\rmd y\rmd x} \\
    & \leq\int_{\abs{x}\leq\delta^{-1}}\abs{u(x)}\abs{\int_{\{z:\abs{z}\geq\delta,
    \abs{x}+\abs{x-z}\leq\delta^{-1}\}}\paren{\frac{1}{\varepsilon^{d+\alpha}} p\big(\frac{z}{\varepsilon}\big) - \frac{k\paren{z/\abs{z}}}{\abs{z}^{d+\alpha}} }\varphi(x-z)\rmd z}\rmd x.
  \end{align*}
  Observe that from \eqref{eq,weak convergence of p and k} and its proof, we deduce that the function
  \begin{equation*}
    g^{\varepsilon}(x):=\abs{\int_{\{z:\abs{z}\geq\delta,
    \abs{x}+\abs{x-z}\leq\delta^{-1}\}}\paren{\frac{1}{\varepsilon^{d+\alpha}} p\big(\frac{z}{\varepsilon}\big) - \frac{k\paren{z/\abs{z}}}{\abs{z}^{d+\alpha}} }\varphi(x-z)\rmd z}
  \end{equation*}
  tends pointwise to zero   as $\varezero$ and 
  is bounded above by $C\delta^{-d-\alpha}$  on the set $\{x:\abs{x}\leq\delta^{-1}\}$. Consequently,  \eqref{eq,limit 1} follows directly from the Lebesgue dominated convergence theorem.
  
  On the other hand, by the uniform boundedness of $\blv$ and Young's inequality for convolution, we derive
  \begin{align*}
   &\abs{\frac{1}{\varepsilon^{d+\alpha}}\iint_{G_1^\delta} p\big(\frac{x-y}{\varepsilon}\big)\blv(x,y)\uep(x)\varphi(y)\rmd y\rmd x-\frac{1}{\varepsilon^{d+\alpha}}\iint_{G_1^\delta} p\big(\frac{x-y}{\varepsilon}\big)\blv(x,y)u(x)\varphi(y)\rmd y\rmd x}\\
   &\lesssim \norm{\frac{1}{\varepsilon^{d+\alpha}}p\big(\frac{\cdot}{\varepsilon}\big)}_{L^1
   (\rd\setminus B_\delta)}\norm{\varphi}_{\ard}\norm{\uep-u}_{L^2(B_{\delta^{-1}})}
   \lesssim \delta^{-\alpha}\norm{\uep-u}_{L^2(B_{\delta^{-1}})},
  \end{align*} 
  where we have used condition \eqref{eq,h2 right} in the last inequality.  Together with the convergence of $\uep$ to $u$ in $\ardloc$, this yields
  \begin{equation}\label{eq,limit 2}
  \begin{aligned}
    &\frac{1}{\varepsilon^{d+\alpha}}\iint_{G_1^\delta} p\big(\frac{x-y}{\varepsilon}\big)\blv(x,y)\uep(x)\varphi(y)\rmd y\rmd x\\
    &=\frac{1}{\varepsilon^{d+\alpha}}\iint_{G_1^\delta} p\big(\frac{x-y}{\varepsilon}\big)\blv(x,y)u(x)\varphi(y)\rmd y\rmd x+o(1),
    \end{aligned}
  \end{equation}
  where $o(1)\to 0$ as $\varezero$. In view of \eqref{eq,limit 1} and \eqref{eq,limit 2}, \eqref{eq,u x and varphi y} follows once we show that 
  \begin{equation}\label{eq,limit 3}
  \begin{aligned}
    &\frac{1}{\varepsilon^{d+\alpha}}\iint_{G_1^\delta} p\big(\frac{x-y}{\varepsilon}\big)\blv(x,y)u(x)\varphi(y)\rmd y\rmd x\\
    &=\frac{\bar{\Lambda}}{\varepsilon^{d+\alpha}}\iint_{G_1^\delta} p\big(\frac{x-y}{\varepsilon}\big)u(x)\varphi(y)\rmd y\rmd x+o(1).
    \end{aligned}
  \end{equation} 
  By the density of $C_c^\infty(B_{\delta^{-1}})$ in $L^2(B_{\delta^{-1}})$, it suffices to prove the above for $u\in C_c^\infty(B_{\delta^{-1}})$.
  
    The proof relies on a partition of the domain $G_1^\delta$ into $\varepsilon$-cubes. Denote $\ikv=\varepsilon(x_k,y_k)+\varepsilon[-1/2,1/2]^{2d}$, where $x_k$, $y_k\in\zd$ and $k\in\n$. Let 
  \begin{equation*}
   \ndv:=\{k\in\n:\ikv\cap G_1^\delta\neq\varnothing\}\quad\text{and}\quad
   G:=\bigcup_{k\in\ndv}\ikv\setminus G_1^\delta.
  \end{equation*}
  Then, it is geometrically clear that for sufficiently small $\varepsilon>0$,
\begin{equation*}
  G_1^\delta\subset\bigcup_{k\in\ndv}\ikv \quad\text{and}\quad G\subset D_1\cup D_2,
\end{equation*}
where
\begin{equation*}
  D_1=\{(x,y):\delta-C(d)\varepsilon\leq\abs{x-y}\leq\delta,
 \abs{x}+\abs{y}\leq\delta^{-1}+C(d)\varepsilon\},
\end{equation*}
\begin{equation*}
  D_2=\{(x,y):\abs{x-y}\geq\delta-C(d)\varepsilon,
  \delta^{-1}\leq\abs{x}+\abs{y}\leq\delta^{-1}+C(d)\varepsilon\}.
\end{equation*}
This together with condition \eqref{eq,h2 right} allows us to show that
\begin{equation}\label{eq,errer of G1 and Ik}
  \frac{1}{\varepsilon^{d+\alpha}}\iint_{G_1^\delta} p\big(\frac{x-y}{\varepsilon}\big)\rmd y\rmd x=\frac{1}{\varepsilon^{d+\alpha}}\sum_{k\in\ndv}\iint_{\ikv} p\big(\frac{x-y}{\varepsilon}\big)\rmd y\rmd x+o(1).
  \end{equation}  
  Indeed,  we have
  \begin{align*}
    \frac{1}{\varepsilon^{d+\alpha}}\iint_{G} p\big(\frac{x-y}{\varepsilon}\big)\rmd y\rmd x
    &\leq\frac{1}{\varepsilon^{d+\alpha}}\iint_{D_1} p\big(\frac{x-y}{\varepsilon}\big)\rmd y\rmd x+\frac{1}{\varepsilon^{d+\alpha}}\iint_{D_2} p\big(\frac{x-y}{\varepsilon}\big)\rmd y\rmd x\\
    &\leq\varepsilon^{d-\alpha}\int_{\delta/\varepsilon-C(d)\leq\abs{z}\leq\delta/\varepsilon}
    p(z)\rmd z\int_{\abs{x}\leq2\delta^{-1}/\varepsilon}\rmd x\\
    &\quad+\varepsilon^{d-\alpha}\int_{\abs{z}\geq\delta/\varepsilon-C(d)}
    p(z)\rmd z\int_{\delta^{-1}/\varepsilon\leq\abs{x}+\abs{x-z}\leq\delta^{-1}/\varepsilon+C(d)}\rmd x.
  \end{align*}
  In view of assumption \eqref{H3}, we deduce that
  \begin{equation*}
    \varepsilon^{-\alpha}\int_{\delta/\varepsilon-C(d)\leq\abs{z}\leq\delta/\varepsilon}
    p(z)\rmd z=\frac{1}{\alpha}\paren{\frac{1}{\paren{\delta-C(d)\varepsilon}^\alpha}-
    \frac{1}{\delta^\alpha}}\int_{\sd}k(s)\rmd S+o(1)=o(1).
  \end{equation*}
  This clearly yields
  \begin{equation*}
    \varepsilon^{d-\alpha}\int_{\delta/\varepsilon-C(d)\leq\abs{z}\leq\delta/\varepsilon}
    p(z)\rmd z\int_{\abs{x}\leq2\delta^{-1}/\varepsilon}\rmd x=o(1).
  \end{equation*} 
  On the other hand, by virtue of assumption \eqref{eq,h2 right}, we derive
  \begin{equation*}
    \int_{\abs{z}\geq\delta/\varepsilon-C(d)}
    p(z)\rmd z\leq\int_{\abs{z}\geq\delta/(2\varepsilon)}
    p(z)\rmd z\lesssim C(\delta)\varepsilon^{\alpha}.
  \end{equation*}
 Observe that for all $z\in\rd$ and any fixed positive constant $C$, the Lebesgue measure of the set 
 $\{x:R\leq\abs{x}+\abs{x-z}\leq R+C\}$ is of order $R^{d-1}$ as $R\to\infty$. Thus, combining this fact with the above estimate, we obtain
 \begin{equation*}
   \varepsilon^{d-\alpha}\int_{\abs{z}\geq\delta/\varepsilon-C(d)}
    p(z)\rmd z\int_{\delta^{-1}/\varepsilon\leq\abs{x}+\abs{x-z}\leq\delta^{-1}/\varepsilon+C(d)}\rmd x
    \lesssim \varepsilon.
 \end{equation*}
  We therefore arrive at
\begin{equation*}
   \frac{1}{\varepsilon^{d+\alpha}}\iint_{G} p\big(\frac{x-y}{\varepsilon}\big)\rmd y\rmd x
    = o(1),
\end{equation*}
which implies \eqref{eq,errer of G1 and Ik}.
 
   We proceed to prove \eqref{eq,limit 3}.   Thanks to \eqref{eq,errer of G1 and Ik} and 
   \eqref{eq:Lamda condition periodic}, we have
   \begin{equation}\label{eq,discrete sum}
     \begin{aligned}
     &\frac{1}{\varepsilon^{d+\alpha}}\iint_{G_1^\delta} p\big(\frac{x-y}{\varepsilon}\big)\blv(x,y)u(x)\varphi(y)\rmd y\rmd x\\
     &=\frac{1}{\varepsilon^{d+\alpha}}\sum_{k\in\ndv}\iint_{\ikv} p\big(\frac{x-y}{\varepsilon}\big)\blv(x,y)u(x)\varphi(y)\rmd y\rmd x+o(1)\\
    &=\frac{1}{\varepsilon^{d+\alpha}}\sum_{k\in\ndv}u(x_k)\varphi(y_k)\iint_{\ikv} p\big(\frac{x-y}{\varepsilon}\big)\blv(x,y)\rmd y\rmd x+o(1).
   \end{aligned}
   \end{equation}
   Let 
 \begin{equation*}
   \tilde{p}_k=\fint_{\ikv}p\big(\frac{x-y}{\varepsilon}\big)\rmd y\rmd x=\iint_{Q_{x_k}\times Q_{y_k}}p(x-y)\rmd y\rmd x.
 \end{equation*}
 Then, on account of the $\varepsilon\zd$-periodicity of $\blv$, we derive
 \begin{equation}\label{eq,fen jie p}
  \iint_{\ikv} p\big(\frac{x-y}{\varepsilon}\big)\blv(x,y)\rmd y\rmd x=\bar{\Lambda}\tilde{p}_k\varepsilon^{2d}+\iint_{\ikv} \big(p\big(\frac{x-y}{\varepsilon}\big)-\tilde{p}_k\big)\blv(x,y)\rmd y\rmd x.
 \end{equation}
 Now let us deal with the latter term on the right-hand side of the above equality.
 By the change of variables $z=(x-y)/\varepsilon$ and in view of  \eqref{eq:Lamda condition periodic}, we have
 \begin{align*}
   &\abs{\iint_{\ikv} \big(p\big(\frac{x-y}{\varepsilon}\big)-\tilde{p}_k\big)\blv(x,y)\rmd y\rmd x}\\
   &\lesssim\varepsilon^{2d}\int_{Q_{y_k}}\int_{Q_{x_k-y}}\abs{p(z)-\int_{Q_{y_k}}
   \int_{Q_{x_k-\tilde{y}}}p(\tilde{z})\rmd \tilde{z}\rmd\tilde{y}}\rmd z\rmd y\\
   &\lesssim\varepsilon^{2d}\int_{Q_{y_k}}\int_{Q_{y_k}}\int_{Q_{x_k-y}}\abs{p(z)-
   \int_{Q_{x_k-\tilde{y}}}p(\tilde{z})\rmd \tilde{z}}\rmd z\rmd y\rmd\tilde{y}.
 \end{align*} 
 Observe that for $y$, $\tilde{y}\in Q_{y_k}$, $\abs{x_k-y-(x_k-\tilde{y})}\leq\sqrt{d}$. Also observe that for all $k\in\ndv$ and $y\in Q_{y_k}$, $\abs{x_k-y}\geq\delta/(2\varepsilon)$, where we use the relation $\varepsilon\ll\delta$. Hence, we can apply assumption \eqref{H4} to deduce that
 \begin{equation*}
   \abs{\iint_{\ikv} \big(p\big(\frac{x-y}{\varepsilon}\big)-\tilde{p}_k\big)\blv(x,y)\rmd y\rmd x}\lesssim\phi(\frac{\delta}{2\varepsilon})\iint_{\ikv} p\big(\frac{x-y}{\varepsilon}\big)\rmd y\rmd x.
 \end{equation*}
 Combing this with \eqref{eq,errer of G1 and Ik}, we obtain
 \begin{equation}\label{eq,p-tilde p}
   \begin{aligned}
     & \abs{\frac{1}{\varepsilon^{d+\alpha}}\sum_{k\in\ndv}u(x_k)\varphi(y_k)\iint_{\ikv}
     \big(p\big(\frac{x-y}{\varepsilon}\big)-\tilde{p}_k\big)\blv(x,y)\rmd y\rmd x} \\
     &\lesssim\frac{1}{\varepsilon^{d+\alpha}}\phi(\frac{\delta}{2\varepsilon})
     \norm{u}_{\linftyrd}
     \norm{\varphi}_{\linftyrd}\sum_{k\in\ndv}\iint_{\ikv} p\big(\frac{x-y}{\varepsilon}\big)\rmd y\rmd x \\
     &=\phi(\frac{\delta}{2\varepsilon})
     \norm{u}_{\linftyrd}
     \norm{\varphi}_{\linftyrd}\paren{\iint_{G_1^\delta} \frac{1}{\varepsilon^{d+\alpha}}p\big(\frac{x-y}{\varepsilon}\big)\rmd y\rmd x+o(1)}
   \end{aligned}
 \end{equation}
 Observe that applying the argument in \eqref{eq,bound on E}, we deduce that
 \begin{equation*}
  \iint_{G_1^\delta} \frac{1}{\varepsilon^{d+\alpha}}p\big(\frac{x-y}{\varepsilon}\big)\rmd y\rmd x\lesssim
   \frac{1}{\varepsilon^{\alpha}}\int_{\abs{x}\leq\delta^{-1}}\rmd x\int_{\abs{z}\geq\delta
   /\varepsilon}p(z)\rmd z\lesssim\delta^{-d-\alpha}.
 \end{equation*}
 Inserting this into \eqref{eq,p-tilde p} and using the assumption $\phi(\frac{\delta}{2\varepsilon})\to0$ as $\varezero$ from \eqref{H4}, \eqref{eq,p-tilde p} becomes
 \begin{equation}\label{eq,p-tilde p last edition}
   \frac{1}{\varepsilon^{d+\alpha}}\sum_{k\in\ndv}u(x_k)\varphi(y_k)\iint_{\ikv}
     \big(p\big(\frac{x-y}{\varepsilon}\big)-\tilde{p}_k\big)\blv(x,y)\rmd y\rmd x=o(1).
 \end{equation}
 The combination of \eqref{eq,discrete sum}, \eqref{eq,fen jie p}, and \eqref{eq,p-tilde p last edition} yields
 \begin{align*}
   \frac{1}{\varepsilon^{d+\alpha}}\iint_{G_1^\delta}& p\big(\frac{x-y}{\varepsilon}\big)\blv(x,y)u(x)\varphi(y)\rmd y\rmd x=\frac{1}{\varepsilon^{d+\alpha}}\sum_{k\in\ndv}u(x_k)\varphi(y_k)
   \bar{\Lambda}\tilde{p}_k\varepsilon^{2d}\\
   &=\frac{\bar{\Lambda}}{\varepsilon^{d+\alpha}}\sum_{k\in\ndv}\iint_{\ikv} u(x)\varphi(y) p\big(\frac{x-y}{\varepsilon}\big)\rmd y\rmd x+o(1)\\
   &=\frac{\bar{\Lambda}}{\varepsilon^{d+\alpha}} \iint_{G_1^\delta}u(x)\varphi(y) p\big(\frac{x-y}{\varepsilon}\big)\rmd y\rmd x+o(1),
 \end{align*}
 where we have used \eqref{eq,errer of G1 and Ik} in the last equality. This establishes \eqref{eq,limit 3} and hence completes the proof of the lemma. 
\end{proof}

\subsection[Homogenization in L2Rd]{Homogenization in $\ard$}\label{subsection3}

In the previous section, we showed that, up to a subsequence, $\uep\to u$ in $\ardloc$ as $\varezero$. In this subsection, we prove that the convergence holds globally in $\ard$. Owing to the uniqueness of the solution to the limiting equation, this will imply the convergence of the whole sequence and thus complete the proof of Theorem \ref{thm:periodic theorem}. The global $\ard$ convergence relies on the following lemma: 
\begin{lemma}\cite[Corollary 4.27]{MR2759829}\label{lemma,Kolmogorov global compactness}
  Under the assumptions of Lemma \ref{lemma,Kolmogorov compactness}, further assume that  for any $\kappa>0$, there exists a bounded and measurable  subset $D\subset\rd$ such that $\norm{u}_{L^q(\rd\setminus D)}<\kappa$ for all $u\in\mathcal{U}$.
  Then   $\mathcal{U}$ has compact closure  in $\qqrd$.
  \end{lemma}
  \begin{proof}[Proof of Theorem \ref{thm:periodic theorem}]
    Let  $\psi_n(x):=\psi(x/n)$ be the scaling of a standard smooth exterior cutoff function $\psi$ supported in $\rd\setminus B_1$ with $\psi\equiv1$ on $\rd\setminus B_2$.
  As a consequence of the above lemma, to prove the global $\ard$ convergence, it suffices to show that 
  \begin{equation}\label{eq,psi uep 2norm}
    \sup_{0<\varepsilon<1}\iintrdd\psi_n(x)(\uep(x))^2\rmd x\to0,\quad\text{as}\quad n\to\infty.
  \end{equation}
  For this purpose, testing equation \eqref{eq,equation of uep} against $\psi_n\uep$, we obtain
  \begin{equation}\label{eq,test by psi uep}
    \begin{aligned}
  &\frac{1}{2\varepsilon^{d+\alpha}}\iintrdd p\big(\frac{x-y}{\varepsilon}\big)\blv(x,y)\paren{\uep(y)-\uep(x)}\paren{\psi_n(y)\uep(y)
  -\psi_n(x)\uep(x)}\rmd y\rmd x\\
  &\quad+ m\intrd \paren{\uep(x)}^2\psi_n(x)\rmd x=\intrd f(x)\psi_n(x)\uep(x)\rmd x.
\end{aligned}
  \end{equation}
  Since $\psi_n\to0$ pointwise as $n\to\infty$, from the Cauchy-Schwarz inequality,   a priori estimate \eqref{eq,estimate of uep}, and the Lebesgue dominated convergence theorem, it follows that as $n\to\infty$,
  \begin{equation}\label{eq,f term tens 0}
    \abs{\intrd f(x)\psi_n(x)\uep(x)\rmd x}\leq\norm{\uep}_{\ard}\norm{f(x)\psi_n(x)}_{\ard}
    \lesssim\norm{f(x)\psi_n(x)}_{\ard}\to 0.
  \end{equation} 
  On the other hand, we rewrite the first integral in \eqref{eq,test by psi uep} as follows:
    \begin{equation}\label{eq,term1}
     \begin{aligned}
  &\frac{1}{\varepsilon^{d+\alpha}}\iintrdd p\big(\frac{x-y}{\varepsilon}\big)\blv(x,y)\paren{\uep(y)-\uep(x)}\paren{\psi_n(y)\uep(y)
  -\psi_n(x)\uep(x)}\rmd y\rmd x\\
  &=\frac{1}{\varepsilon^{d+\alpha}}\iintrdd p\big(\frac{x-y}{\varepsilon}\big)\blv(x,y)\paren{\uep(y)-\uep(x)}^2\psi_n(y)\rmd y\rmd x\\
  &\quad+\frac{1}{\varepsilon^{d+\alpha}}\iint_{\{(x,y):\abs{x-y}<M\varepsilon\}} p\big(\frac{x-y}{\varepsilon}\big)\blv(x,y)\paren{\uep(y)-\uep(x)}\paren{\psi_n(y)
  -\psi_n(x)}\uep(x)\rmd y\rmd x\\
  &\quad+\frac{1}{\varepsilon^{d+\alpha}}\iint_{\{(x,y):\abs{x-y}\geq M\varepsilon\}} p\big(\frac{x-y}{\varepsilon}\big)\blv(x,y)\paren{\uep(y)-\uep(x)}\paren{\psi_n(y)
  -\psi_n(x)}\uep(x)\rmd y\rmd x\\
  &=:I_1+I_2+I_3.
\end{aligned} 
    \end{equation}
    It is obvious that
\begin{equation}\label{eq,I1}
  I_1\geq0.
\end{equation}
In view of assumptions \eqref{eq,h1} and \eqref{eq:Lamda condition periodic},  the Cauchy-Schwarz inequality, a priori estimates \eqref{eq,p uep estimate} and \eqref{eq,estimate of uep} as well as the fact that $\abs{\psi_n(x)-\psi_n(y)}\leq C(d)n^{-1}\abs{x-y}$,   we derive
\begin{equation}\label{eq,I2}
  \begin{aligned}
  \abs{I_2}&\lesssim\paren{\frac{1}{\varepsilon^{d+\alpha}}\iintrdd p\big(\frac{x-y}{\varepsilon}\big)\paren{\uep(y)-\uep(x)}^2\rmd y\rmd x}^{1/2}\\
  &\quad\quad\times\paren{\frac{1}{\varepsilon^{d+\alpha}}\iint_{\{(x,y):\abs{x-y}<M\varepsilon\}} p\big(\frac{x-y}{\varepsilon}\big)\abs{\psi_n(y)
  -\psi_n(x)}^2\abs{\uep(x)}^2\rmd y\rmd x}^{1/2}\\
  &\lesssim n^{-1}\varepsilon^{1-\frac{\alpha}{2}}\paren{\intrd p(z)\rmd z\intrd\paren{\uep(x)}^2\rmd x}^{1/2}\lesssim n^{-1}.
  \end{aligned}
\end{equation}
 Similarly, we also have
\begin{equation*}
  \abs{I_3}\lesssim\paren{\frac{1}{\varepsilon^{d+\alpha}}\iint_{\{(x,y):\abs{x-y}\geq M\varepsilon\}} p\big(\frac{x-y}{\varepsilon}\big)\abs{\psi_n(y)
  -\psi_n(x)}^2\abs{\uep(x)}^2\rmd y\rmd x}^{1/2}
\end{equation*}
Using the fact that $\abs{\psi_n(y)-\psi_n(x)}\leq C(d)\min\{n^{-1}\abs{y-x},1\}$, the above inequality becomes
\begin{align*}
  \abs{I_3}&\lesssim\paren{\frac{1}{n^2\varepsilon^{d+\alpha}}\iint_{\{(x,y): M\varepsilon\leq\abs{x-y}\leq n\}} p\big(\frac{x-y}{\varepsilon}\big)\abs{x-y}^2\abs{\uep(x)}^2\rmd y\rmd x}^{1/2}\\
  &\quad\quad+\paren{\frac{1}{\varepsilon^{d+\alpha}}\iint_{\{(x,y): \abs{x-y}\geq n\}} p\big(\frac{x-y}{\varepsilon}\big)\abs{\uep(x)}^2\rmd y\rmd x}^{1/2}\\
  &\lesssim n^{-1}\paren{\varepsilon^{2-\alpha}\int_{\{z: M\leq\abs{z}\leq n/\varepsilon\}}p(z)\abs{z}^2\rmd z}^{1/2}+\paren{\varepsilon^{-\alpha}\int_{\{z: \abs{z}\geq n/\varepsilon\}}p(z)\rmd z}^{1/2}.
\end{align*}
Repeating the argument in \eqref{eq, estimate pz z2}, we have
\begin{equation*}
 \int_{\{z: M\leq\abs{z}\leq n/\varepsilon\}}p(z)\abs{z}^2\rmd z\lesssim
  \paren{\frac{n}{\varepsilon}}^{2-\alpha}\quad\text{and}\quad
  \int_{\{z: \abs{z}\geq n/\varepsilon\}}p(z)\rmd z\lesssim\paren{\frac{n}{\varepsilon}}^{-\alpha}.
\end{equation*}
We therefore obtain 
\begin{equation}\label{eq,I3}
  \abs{I_3}\lesssim n^{-\frac{\alpha}{2}}.
\end{equation}
Putting together \eqref{eq,test by psi uep}--\eqref{eq,I3}, we arrive at \eqref{eq,psi uep 2norm}. This completes the proof.
  \end{proof}
 \subsection{Proof of Theorem \ref{thm:locally periodic theorem}} 
 The proof of Theorem \ref{thm:locally periodic theorem} follows the same argument as that of Theorem \ref{thm:periodic theorem} developed in Subsections \ref{subsection1}--\ref{subsection3}, with minor modifications. Specifically, the two key formulas  \eqref{eq,discrete sum} and \eqref{eq,fen jie p} are replaced by 
     \begin{align*}
     &\frac{1}{\varepsilon^{d+\alpha}}\iint_{G_1^\delta} p\big(\frac{x-y}{\varepsilon}\big)\blv(x,y)u(x)\varphi(y)\rmd y\rmd x\\
    &=\frac{1}{\varepsilon^{d+\alpha}}\sum_{k\in\ndv}u(x_k)\varphi(y_k)\iint_{\ikv} p\big(\frac{x-y}{\varepsilon}\big)\Lambda(x_k,y_k,x/\varepsilon,y/\varepsilon)\rmd y\rmd x+o(1),
   \end{align*}
   and
    \begin{align*}
  &\quad\iint_{\ikv} p\big(\frac{x-y}{\varepsilon}\big)\Lambda(x_k,y_k,x/\varepsilon,y/\varepsilon)\rmd y\rmd x\\
  &=\bar{\Lambda}(x_k,y_k)\tilde{p}_k\varepsilon^{2d}+\iint_{\ikv} \big(p\big(\frac{x-y}{\varepsilon}\big)-\tilde{p}_k\big)\Lambda(x_k,y_k,x/\varepsilon,y/\varepsilon)\rmd y\rmd x.
 \end{align*}
 These substitutions are justified by the equicontinuity  \eqref{eq, modulus of continuity} and  $\zd$-periodicity \eqref{eq,symmetric of locally periodic} of the coefficient $\Lambda(x,y,\xi,\eta)$. This replacement allows us to derive, in the present locally periodic
 coefficient setting, an analogue of equality \eqref{eq,limit 3}, namely,
 \begin{align*}
    &\frac{1}{\varepsilon^{d+\alpha}}\iint_{G_1^\delta} p\big(\frac{x-y}{\varepsilon}\big)\blv(x,y)u(x)\varphi(y)\rmd y\rmd x\\
    &=\frac{1}{\varepsilon^{d+\alpha}}\iint_{G_1^\delta} p\big(\frac{x-y}{\varepsilon}\big)\bar{\Lambda}(x,y)u(x)\varphi(y)\rmd y\rmd x+o(1).
    \end{align*}
 We note that all other steps remain valid under the symmetry and uniform boundedness assumptions (see \eqref{eq,locally periodic}) of $\Lambda(x,y,\xi,\eta)$. Therefore, with these two substitutions, the rest of the argument carries over with the same strategy.
\section*{Acknowledgments} 

The work of the third author is partially supported by Young Scientific and Technological Talents (Level Three) in Tianjin and Tianjin Natural Science Foundation Project (No. 25JCQNJC01400) and the work of the corresponding author is partially supported by National Natural Science Foundation of China (NSFC Grant No. 12571103).
%\printbibliography

\bibliography{ref4}

\end{document}